\documentclass[]{amsart}
\usepackage[utf8]{inputenc}
\usepackage[mathscr]{eucal} % So-called Euler fonts, allows \mathscr
\usepackage{stmaryrd} % St. Mary's font, fixes a few bad symbols
\usepackage{amsmath,amsthm,amsfonts,amssymb,amscd} % Standard AMS packages
\usepackage[dvipsnames]{xcolor} % Allows colors
\usepackage{xspace} % Allows smart spacing after period
\usepackage{fancyhdr} % For footers, headers
\usepackage{graphicx} % For inserting images
\usepackage{listings} % For source code listing
\usepackage[]{hyperref} % For hyperlinks
\usepackage{enumitem} % More control for enumerated lists
\usepackage{relsize} % Change font size relative to current size
\usepackage{tikz-cd} % For commutative diagrams
\usepackage{mdwlist} % Additional list customization
\usepackage{multicol} % Allows multiple columns
\usepackage{float} % Improved placement for floating objects
\usepackage{adjustbox} % Adjust boxed content
\usepackage{tikz} % Regular tikz
\usepackage[noabbrev,nameinlink]{cleveref} % Additional reference customization
\Crefformat{section}{#2\S#1#3}% % to get \S instead of \Section
\usepackage{setspace} % Used for changing spacing
\usepackage{bbm} % more bb
\usepackage{mathtools}
\usepackage[all]{xy}
\usetikzlibrary{matrix, calc, arrows}

\hypersetup{
    colorlinks=true, %set true if you want colored links
    linktoc=all,     %set to all if you want both sections and subsections linked
    linkcolor=Brown,citecolor=Brown,urlcolor=MidnightBlue,  %choose some color if you want links to stand out u
}
\setlist[enumerate]{label=(\alph*)} % Default enumerate alpha
\usetikzlibrary{graphs,decorations.pathmorphing,decorations.markings}
\tikzcdset{scale cd/.style={every label/.append style={scale=#1},
        cells={nodes={scale=#1}}}}

\newcommand{\Aut}{\operatorname{Aut}}

\newcommand{\inv}{^{-1}}

\newcommand{\End}{\operatorname{End}}

\newcommand{\Res}{\operatorname{Res}}
\newcommand{\Ind}{\operatorname{Ind}}
\newcommand{\Inf}{\operatorname{Inf}}

\newcommand{\im}{\operatorname{im}}

\newcommand{\Irr}{\operatorname{Irr}}
\newcommand{\Gal}{\operatorname{Gal}}

\newcommand{\on}[1]{\operatorname{#1}}

\newcommand{\calF}{\mathcal{F}}
\newcommand{\calG}{\mathcal{G}}

\newcommand{\calO}{\mathcal{O}}

\newcommand{\bbC}{\mathbb{C}}

\newcommand{\bbF}{\mathbb{F}}

\newcommand{\bbH}{\mathbb{H}}

\newcommand{\bbN}{\mathbb{N}}

\newcommand{\bbQ}{\mathbb{Q}}
\newcommand{\bbR}{\mathbb{R}}

\newcommand{\bbZ}{\mathbb{Z}}

\makeatletter
\newcommand*{\doublerightarrow}[2]{\mathrel{
        \settowidth{\@tempdima}{$\scriptstyle#1$}
        \settowidth{\@tempdimb}{$\scriptstyle#2$}
        \ifdim\@tempdimb>\@tempdima \@tempdima=\@tempdimb\fi
        \mathop{\vcenter{
                \offinterlineskip\ialign{\hbox to\dimexpr\@tempdima+1em{##}\cr
                    \rightarrowfill\cr\noalign{\kern.5ex}
                    \rightarrowfill\cr}}}\limits^{\!#1}_{\!#2}}}
\newcommand*{\triplerightarrow}[1]{\mathrel{
        \settowidth{\@tempdima}{$\scriptstyle#1$}
        \mathop{\vcenter{
                \offinterlineskip\ialign{\hbox to\dimexpr\@tempdima+1em{##}\cr
                    \rightarrowfill\cr\noalign{\kern.5ex}
                    \rightarrowfill\cr\noalign{\kern.5ex}
                    \rightarrowfill\cr}}}\limits^{\!#1}}}
\makeatother

\newcommand{\RO}{\on{RO}}
\newcommand{\CF}{\on{CF}}

\newtheorem{theorem}{Theorem}[section]
\newtheorem{lemma}[theorem]{Lemma}
\newtheorem{proposition}[theorem]{Proposition}
\newtheorem{corollary}[theorem]{Corollary}

\newtheorem*{theorem*}{Theorem}
\newtheorem*{conjecture*}{Conjecture}
\newtheorem{conjecture}{Conjecture}

\newtheorem{question}{Question}

\theoremstyle{definition}
\newtheorem*{definition*}{Definition}

\theoremstyle{remark}

\newtheorem{remark}[theorem]{Remark}
\newtheorem{example}[theorem]{Example}

\theoremstyle{definition}
\newtheorem{definition}[theorem]{Definition}

\makeatletter
\newcommand{\xhookdoubleheadrightarrow}[2][]{%
    \lhook\joinrel
    \ext@arrow 0359\rightarrowfill@ {#1}{#2}%
    \mathrel{\mspace{-15mu}}\rightarrow
}
\makeatother

\theoremstyle{theorem}
\newtheorem{innercustomthm}{Theorem}
\newenvironment{customthm}[1]
{\renewcommand\theinnercustomthm{#1}\innercustomthm}
{\endinnercustomthm}

\AddToHook{env/corollary/begin}{\crefalias{theorem}{corollary}}
\AddToHook{env/proposition/begin}{\crefalias{theorem}{proposition}}
\AddToHook{env/lemma/begin}{\crefalias{theorem}{lemma}}
\AddToHook{env/definition/begin}{\crefalias{theorem}{definition}}
\AddToHook{env/example/begin}{\crefalias{theorem}{example}}
\AddToHook{env/remark/begin}{\crefalias{theorem}{remark}}
\AddToHook{env/observation/begin}{\crefalias{theorem}{observation}}
\AddToHook{env/construction/begin}{\crefalias{theorem}{construction}}
\AddToHook{env/question/begin}{\crefalias{theorem}{question}}

\begin{document}
    \title{Non-orientable representation spheres}
    \author{Sam K. Miller}
    \address{Department of Mathematics, University of Georgia, Athens GA 30602, United States of America} %required
    \email{sam.miller@uga.edu} %optional
     %required
     %optional
    \begin{abstract}
        For any finite group, we pose the following question: given an irreducible real representation, is the associated representation sphere non-orientable if and only if the representation is nontrivial of real type? We prove that the answer to this question is yes if the group has a normal Sylow 2-subgroup, but exhibit a 2-nilpotent group of order 112 for which the answer is no via elementary arguments. We also link the question to the unit group of the Burnside ring, where we recover a basis discovered by Bouc for 2-groups, and pose a conjecture about detection of non-orientability from solvable subgroups. The question is motivated by issues arising from the construction of permutation twisted cohomology for finite groups.
    \end{abstract}

    \maketitle

    \section*{Introduction}

    In their landmark paper outlining the theory of homotopy representations for finite groups, tom Dieck--Petrie defined the notion of an \textit{orientable} homotopy representation  \cite{tDP82}. Perhaps one of the easiest examples of a homotopy representation is a \emph{representation sphere}, i.e., the one-point compactification of a real representation, by Illman's equivariant triangulation theorem \cite{Ill78}. However, the question of orientability of representation spheres appears to have never been studied. It is not hard to prove that the only possible candidates for non-orientable spheres arising from irreducible real representations are those which are nontrivial and of real type. This motivates the following question:

    \begin{question}\label{q:the_question}
        Let $G$ be a finite group and $V$ be an irreducible real representation. Is the representation sphere $S^V$ non-orientable if and only if $V$ is nontrivial and has real type?
    \end{question}

    We say any group $G$ for which the answer of \Cref{q:the_question} is yes is \emph{nirrorno} (nontrivial irreducible real representations of real-type are non-orientable). It turns out that perhaps surprisingly, many finite groups are nirrorno.

    \begin{customthm}{A}\label{thm:A}(\Cref{thm:2_grp_nirrorno})
        If $G$ has a normal Sylow 2-subgroup, then $G$ is nirrorno.
    \end{customthm}

    To answer \Cref{q:the_question}, we prove an equivalent formulation of non-orientability of a representation sphere (equivalently, a real representation) which brings the unit group $A(G)^\times$ of the Burnside ring of a finite group $G$ into the picture. We also deduce a criterion on solvable subgroups of $G$ which can detect non-orientability.

    \begin{customthm}{B}(\Cref{prop:equiv_conditions}, \Cref{prop:solv_cond})
        Let $V$ be a real representation of a finite group $G$, and let $\pi_V(H) := \dim V^H \on{mod} 2$ denote the \emph{parity function} of $V$. The following are equivalent.
        \begin{enumerate}
            \item There exists a subgroup $H$ of $G$ and $g \in N_G(H)$ such that $\det(\rho_{V^H}(g)) = -1$, i.e., $S^V$ is non-orientable;
            \item There exist subgroups $K_1 \trianglelefteq K_2 \leq G$ with $[K_2:K_1] = 2$ and $\dim_V(K_1) \not\equiv \dim_V(K_2) \mod 2$.
        \end{enumerate}
        Moreover, if $\pi_V$ is nonconstant on the set of solvable subgroups of $G$, then the equivalent conditions of \Cref{prop:equiv_conditions} hold for $V$. One may choose $(K_1,K_2)$ to be solvable in this case.
    \end{customthm}

    We remark that the parity function $\pi_V$ identifies with the image of $V$ in $A(G)^\times$ under the \emph{tom Dieck homomorphism}. Whether the converse of the final statement holds is open, and we conjecture it holds.

    \begin{conjecture*}(\Cref{q:solv_cond})
        If $S^V$ is non-orientable, then $\pi_V$ is nonconstant on the set of solvable subgroups of $G$.
    \end{conjecture*}

    In a computational test using GAP \cite{GAP4} for (non-solvable) groups of order up to 2000, as well as all 88 finite simple groups which the library \texttt{TomLib} \cite{TomLib1.2.11} has data of, we found no counterexample. See the appendix for all relevant code. We remark on the apparent difficulty of proving this conjecture, see \Cref{rmk:conj_hard}.

    The fact that solvable subgroups seem to detect non-orientability on the nose motivates the following definition.

    \begin{definition*}(\Cref{def:strong_nirrorno})
        Let $\calG(G)$ denote a set of representatives of Galois orbits of nontrivial irreducible real representations of real type. An even order finite group $G$ is \emph{strong nirrorno} if the following condition holds: the set \[\{\pi_\bbR\}\cup \bigcup_{V \in \calG(G)}\{\pi_V\}\subseteq \CF(\on{Solv}(G), \bbF_2)\] restricted to the solvable subgroups of $G$ is linearly independent.
    \end{definition*}

    Strong nirrorno-ness implies nirrorno-ness and holds for groups with normal Sylow 2-subgroups. Moreover, strong nirrorno-ness is necessary to prove \Cref{thm:A} in generality. Given a 2-group $G$, Bouc deduced a `canonical' basis of the unit group of the Burnside ring $A(G)^\times$. The set listed above is precisely this basis: in other words, one obtains the basis via the tom Dieck homomorphism applied to irreducible non-orientable representation spheres. See \Cref{rmk:burnside} and \Cref{rmk:2_grp_rmks} for details.

    On the other hand, we find that not all groups are nirrorno, and that nirrorno-ness and strong nirrorno-ness are not the same. In fact, the non-nirrorno group we find is 2-nilpotent, i.e., has a normal 2-complement. In particular, it is solvable.

    \begin{customthm}{C}(\Cref{thm:counterex}, \Cref{thm:str_nir_not_nir})
        \begin{enumerate}
            \item The 2-nilpotent group $C_7 \rtimes SD_{16}$, where $SD_{16}$ acts by involutions on $C_7$ with kernel its unique subgroup isomorphic to $Q_8$, contains a real irreducible representation of real type of degree 4 whose associated representation spheres is orientable.
            \item The solvable group $C_3 \rtimes \on{GL}_2(\bbF_3)$ is nirrorno, but not strong nirrorno. This is the minimal example of such a group.
        \end{enumerate}

    \end{customthm}

    The proof of (a) is entirely non-computational, and relies on machinery from twisted Frobenius-Schur indicators. On the other hand, (b) was computed purely computationally via GAP, and we provide all supplemental code in the appendix. Via GAP computation, we also determine that only 3 of the 88 finite simple groups in \texttt{TomLib} are not nirrorno: $J_2, J_3,$ and $\on{PSp}_4(\bbF_5)$.

    \subsection*{Relation to prior work} \Cref{q:the_question} is motivated from the construction of \emph{permutation twisted cohomology} (in the sense of Balmer--Gallauer \cite{BG25}, see also \cite{Mil26}) to understand the tensor-triangular geometry of the tensor-triangulated category $\on{K}_b(\mathsf{perm}(kG))^\natural)$, the bounded homotopy category of $p$-permutation $kG$-modules, for $G$ a finite group and $k$ a field of prime characteristic $p$. Here, one considers an analogous notion of orientability of an \emph{endotrivial complex}, i.e., an invertible object in $\on{K}_b(\mathsf{perm}(kG))^\natural)$.

    Representation spheres induce endotrivial complexes in a precise way, see \cite[Section 3]{Mil26}, and in constructing twisted cohomology, one should only twist by certain `locally trivial' endotrivials, a (possibly) smaller class than orientable endotrivials. In this setting, orientable representation spheres induce locally trivial endotrivials, whereas non-orientable representation spheres induce locally trivial endotrivials only if $p = 2$. Therefore, in the situation of $p$ odd and $G$ a non-2-group of even order, one benefits from knowing when certain representations are orientable or not. The situation is more delicate for representation spheres, see \Cref{ex:endotriv_ex}.

    \subsection*{Statement on generative AI use}

    This paper partially arose from testing the capabilities of a publicly available generative AI model and seeing if it could settle \Cref{q:the_question}, which I had posed as a conjecture for all finite groups. The model was able to reprove some results I had obtained previously and made a few minor discoveries. After more attempts at proving the supposed conjecture, it discovered the counterexample and provided the idea of the proof given. The model also aided in the writing of the GAP code provided in the appendix. All other writing in the paper is entirely my own, and I take full responsibility for its content and correctness.

    \section{Preliminaries}

    For this paper, we adopt the convention of discussing orientability in terms of real representations, as character theory will come into play in the sequel. Throughout, given a $kG$-module $V$ (often referred to as a $k$-representation), for $G$ a finite group and $k$ a characteristic 0 field, $\chi_V\colon G \to k$ denotes the associated character and $\rho_V\colon G \to \Aut(V)$ denotes the associated matrix representation. The subscripts may be removed if the context is clear.

    \begin{definition}
        Let $G$ be a finite group and $V$ be a real $\bbR G$-module (equivalently, an orthogonal representation). We say $V$ is \emph{orientable} if its representation sphere $S^V$ is orientable in the sense of \cite[Definition 1.8]{tDP82}. That is, for every subgroup $H$ of $G$ and $g \in N_G(H)$, the action of $g$ on $(S^V)^H$ is orientation-preserving, i.e., the induced action on $\on{H}^{\dim V^H}((S^V)^H) \cong \bbZ$ is trivial. Equivalently, for all $H$ of $G$ and $g \in N_G(H)$, $\det(\rho_{V^H}(g)) = 1$, where $\rho_{V^H}$ denotes the orthogonal representation associated to $V^H$. Note $\det(\rho_{V^H}(g)) \in \{\pm1\}$.
    \end{definition}

    \begin{example}
        If $G$ has odd order, then every representation of $V$ is orientable, as every element $g \in G$ has odd order.
    \end{example}

    \begin{definition}
        We say an irreducible real representation $V$ is of \emph{real} (resp. \emph{complex}, \emph{quaternion}) type if $\End(V) \cong \bbR$ (resp. $\bbC$, $\bbH$).
        If $V$ is of real type, then $\bbC \otimes_\bbR V$ remains irreducible, but if $V$ is of complex or quaternionic type, then $\bbC \otimes_\bbR V = W \oplus \overline{W}$.

        Related, let $V$ be an irreducible complex representation of $G$ with character $\chi$. The \emph{Frobenius--Schur indicator} $\nu(\chi)$ is defined as follows: \[\nu(\chi) := \frac{1}{|G|}\sum_{g \in G}\chi(g^2)\in \{-1,0,1\}.\] If $\nu(\chi) = 0$, then $\chi$ admits complex values, and $\chi + \overline{\chi}$ is the complexification of a real irreducible representation. if $\nu(\chi) = 1$ then $\chi$ is real-valued and $V$ is the complexification of a real irreducible representation, and if $\nu(\chi) = -1$, then $\chi$ is real-valued, and $\chi+ \chi$ (but not $\chi$) is the complexification of a real irreducible representation.
    \end{definition}

    Complex and quaternion type representations are automatically orientable.

    \begin{proposition}\cite[Proposition 8.4]{GM26}
        Let $V$ be a real irreducible representation of $G$. If $V$ is of complex or quaternionic type, then $V$ is orientable.
    \end{proposition}
    \begin{proof}
        If $V$ is of complex or quaternionic type, then for any subgroup $H$ of $G$, $\bbC \otimes_\bbR V^H = W \oplus \overline{W}$. Then for any $g \in N_G(H)/H$, \[\det(\rho_{V^H}(g)) = \det(\rho_W(g))\det(\rho_{\overline{W}}(g)) = \det(\rho_W(g))\overline{\det(\rho_W(g))}.\] If $V$ is quaternionic type, then \[\det(\rho_W(g))\overline{\det(\rho_W(g))} = |\det(\rho_W(g))|^2 = 1,\] and if $V$ is quaternionic, then since $W = \overline{W}$, \[\det(\rho_W(g))\overline{\det(\rho_W(g))} = \det(\rho_W(g))^2 = 1.\]

    \end{proof}

    % Therefore, the point is moot when $G$ has odd order, its only irreducible real representation of real type is $\bbR$. This fact is well-known, but we provide a short proof for reference.

    % \begin{proposition}\label{prop:no_odd_nirrors}
        %     Let $G$ be an odd order finite group and $V$ be a nontrivial irreducible real representation of $G$. Then $V$ is not of real type. Consequently, all irreducible real representations of $V$ are or
        % \end{proposition}
    % \begin{proof}
        %     We have by \cite[Borel--Smith functions]{tD87} that $\dim V$ is an even-valued function, since $\dim V^G = 0$; this follows easily from the Borel--Smith conditions. However $\bbC \otimes_\bbR V$ is either a direct sum of two irreducibles or an irreducible itself, however by character theory every irreducible character of $G$ must have odd degree, so $\bbC \otimes_\bbR V$, which has even $\bbC$-dimension, cannot be irreducible. Thus $V$ cannot have real type.
        % \end{proof}

    Additionally, the trivial representation $\bbR$, is obviously orientable of real type. So the only irreducible real representations which can possibly be non-orientable are nontrivial irreducible real representations. It is well known that these only occur for groups of even order.

    \begin{definition}
        We abbreviate `nontrivial irreducible real representation of real-type' by \emph{nirror}. We say an finite group $G$ for which the nirrors are precisely the non-orientable irreducibles is \emph{nirrorno} (nirrors are non-orientable). We write $\mathsf{nir}(G)$ for the collection of nirrors of $G$, and write $\calG(G) := \mathsf{nir}(G)/\Gal(\bbQ^{\on{ab}}/\bbQ)$ for (a class of representatives of) the Galois orbits of nirrors given by the action of the absolute Galois group $\Gal(\bbQ^{\on{ab}}/\bbQ)$.
    \end{definition}

    \subsection{Variations on orientability}

    We can consider weaker versions of `orientability' that seem reasonable or are motivated. For instance, if we only look at the orientation behavior for the representation $V$ and not the fixed-point representations $V^H$, one may ask the same question.

    \begin{definition}\label{def:triv_global}
        Say a real representation $V$ of a finite group $G$ has \emph{trivial global orientation behavior} if for all $g \in G$, $\det(\rho_V(g)) = 1$.
    \end{definition}

    In this case, there are a wealth of representations for which the action of all $g$ on $V$ are orientation-preserving, but $V$ is not orientable.

    \begin{example}\label{ex:SO_ex}
        Let $G= A_4$. There is a nirror $\chi_V$ of degree 3, whose associated real representation $V$ arises from the action of $S_4$ on $\bbR^4$, then quotienting by the invariant 1-dimensional subspace. Then $A_4$ is precisely the subgroup of $S_4$ which acts on $V$ by orientation-preserving elements, i.e., the associated representation is of the form $\chi\colon A_4 \to \on{SO}(V)$.

        However, $V$ is non-orientable. Indeed, let $C \cong C_2$ be a subgroup of $A_4$ and let $N \cong V_4$ denote its normalizer. Then $V^{C}$ is a $\bbR[N/C]$-module of dimension 1, and $g \in N \setminus C$ acts by involution; one may see this by deducing $V^N = 0$.
    \end{example}

    The question of orientability originally arose from the construction of \emph{permutation twisted cohomology} for arbitrary finite groups $G$ over a field $k$ of positive characteristic (see \cite{BG25, Mil26} and \cite[Section 8]{GM26}). In this setting, one wants to consider a notion of orientability for an \emph{endotrivial complex}, i.e., an invertible object in the bounded homotopy category of $p$-permutation $kG$-modules, $\on{K}_b(\mathsf{perm}(kG)^\natural)$. See \cite[Definition 8.2]{GM26} for a precise definition. Endotrivial complexes are induced from representation spheres, and if $G$ is a $p$-group, every endotrivial is induced from a virtual representation sphere \cite{M24c}. Moreover, orientable representation spheres induce orientable endotrivial complexes, but the converse need not hold.

    To construct twisted cohomology, one wants to twist by certain `locally trivial' endotrivial complexes, and orientability is a necessary condition for local triviality. However, in the algebraic setting, one can only take fixed-point complexes (via \emph{Brauer quotients/modular fixed points}) at $p$-subgroups, hence orientation behavior can only be detected $p$-locally. More generally, if $k$ is a commutative ring and $p = 1 + \cdots + 1 $ is a nonunit in $k$, then one can take \emph{generalized modular fixed points}, see \cite{GM26}, and so one can consider the orientation behavior at prime-power subgroups. Again, one can find examples of nirrors which have trivial orientation behavior at all $p$-subgroups, for a fixed prime $p$.

    \begin{example}\label{ex:endotriv_ex}
        The same example from \Cref{ex:SO_ex}, of $G = A_4$ and $V$ the irreducible real representation of real type and dimension 3, has trivial orientation behavior on $3$-subgroups. Indeed, from before, we deduced that $\chi\colon A_4 \to SO(V)$ preserves orientation trivially, and the only 3-subgroups of $A_4$ are self-normalizing.
        In this case, the non-orientable $V$ produces an endotrivial complex that is orientable.
    \end{example}

    \section{Equivalent formulations of non-orientability}

    We reformulate non-orientability in terms of elements in the unit group of the Burnside ring of a finite group.

    \begin{definition}
        Given a real representation $V$ of $G$, its \emph{dimension function} $\dim_V \in \CF(G)$ is the \emph{superclass function} (i.e., function from the set of subgroups of $G$ to $\bbZ$ constant on conjugacy classes) \[H \mapsto \dim V^H.\] This is a monotone \emph{Borel--Smith function}, see \cite[Chapter III, Section 5]{tD87}; we discuss further aspects in \Cref{rmk:conj_hard}.
        Write $\pi_V$ for the $\bbF_2$-reduction of $\dim_V$. We call this the \emph{parity function} of $V$.
    \end{definition}

    \begin{remark}\label{rmk:burnside}
        The \emph{Burnside ring} $A(G)$ of a finite group $G$ is the split Grothendieck ring of the category of finite $G$-sets. The problem of determining $A(G)^\times$ for arbitrary finite groups is famously difficult. Indeed, the statement ``if $|G|$ is odd, then $A(G)^\times = \{\pm[G/G]\}$'' is equivalent to the Feit--Thompson odd order theorem, see \cite[Remark 11.2.3]{Bou10}.

        There is an injective ring homomorphism, the \emph{mark homomorphism} \[m_G\colon A(G) \to \on{CF}(G), \quad x \mapsto \big(H \mapsto |x^H|\big),\] and in this way, elements of $A(G)$ are identified via their marks. The \emph{tom Dieck homomorphism} \[\Theta_G\colon \RO(G) \to A(G)^\times, \quad V \mapsto \big(H \mapsto (-1)^{\dim_V(H)}\big)\] is well-defined and surjective when $G$ is a $p$-group \cite{Tor84,Y05}. Moreover, we have an obvious group identification $\CF(G,\bbF_2) \cong \CF(G)^\times$. In this way, the parity functions $\pi_V \in \CF(G, \bbF_2)$ are identified in $\CF(G)^\times$ as nothing more than $\Theta_G(V) \in A(G)^\times \subseteq \CF(G)^\times$. Given this identification, we will regard $\pi_V \in A(G)^\times$ without further mention; we regard the parity function as equivalently, the image of $V$ under the tom Dieck homomorphism.

        The inclusion $m_G\colon A(G)^\times \to \CF(G)^\times$ need not be surjective. Yoshida \cite{Yos90} proves that a function $f \in \CF(G)^\times$ is in the image of $m_G$ if and only if $f$ satisfies the following property: for every $H \leq G$, the \emph{Yoshida} map \[N_G(H) \to \bbZ^\times, \quad g \mapsto f(\langle H,g \rangle)/f(H)\] is a group homomorphism.

        In our case, we have an equality \[\det(\rho_{V^H}(g)) = |\Theta_G(V)^H|/|\Theta_G(V)^{\langle H, g\rangle}| = (-1)^{\dim_V(H) - \dim_V(\langle H,g\rangle)}.\] Indeed, the right-hand side counts the sign of the dimension of an orthogonal complement of the $g$-fixed subspace of $V^H$. The eigenspace corresponding to roots of unity greater than 2 must be even-dimensional, since these are complex and therefore come in conjugate pairs. Therefore, the sign is determined by the dimension of the $-1$-eigenspace, which the determinant counts.
        Hence the collection of Yoshida's maps is precisely the orientation behavior of $V$.

    \end{remark}

    We note that Galois conjugacy does not change the dimension function of a real representation, nor its type.

    \begin{proposition}\label{prop:galois_invariance}
        Let $V$ be a real representation with character $\rho_V$ and $H$ be a subgroup of $G$.
        \begin{enumerate}
            \item We have \[\dim V^H = \langle \Res^G_H \chi_V, 1_{H}\rangle_H = \frac{1}{|H|}\sum_{g \in H} \chi_V(g) = \langle \chi_V, \chi_{\bbR[G/H]}\rangle_G;\]
            \item Let $\sigma$ be a field automorphism of a splitting field. Then $\dim V^H = \dim {}^\sigma V^H$. In particular, $\dim V$ depends only on the Galois class (equivalently, Adams operation class) of $V$.
        \end{enumerate}
    \end{proposition}
    \begin{proof}
        Observe that $V^H$ as vector space is equivalently the subspace of $\Res^G_H V$ isomorphic to a direct sum of trivial representations, from which (a) follows. For (b), note the action of $\sigma$ commutes with both $\langle -, -\rangle$ and $\nu(-)$, and since their outputs are integer-valued, $\sigma$ does not change the output.

    \end{proof}

    \begin{theorem}\label{prop:equiv_conditions}
        Let $V$ be a real representation of a finite group $G$. The following are equivalent.
        \begin{enumerate}
            \item There exists a subgroup $H$ of $G$ and $g \in N_G(H)$ such that $\det(\rho_{V^H}(g)) = -1$;
            \item There exist subgroups $K_1 \trianglelefteq K_2 \leq G$ with $[K_2:K_1] = 2$ and $\pi_V(K_1) \not\equiv \pi_V(K_2)$.
        \end{enumerate}
    \end{theorem}
    \begin{proof}
        One may observe that this follows from the identity in \Cref{rmk:burnside}, however we give an independent proof for completeness.

        (a) $\implies$ (b): Let $H$ and $g \in N_G(H)$ satisfy $\det(\rho_{V^H}(g)) = -1$. Then $g\not\in H$. Write $g = g_2g_{2'}$, where $g_2,g_{2'}$ are the unique commuting $2$- and $2'$-elements respectively. Both are in fact powers of $g$, so both $g_2, g_{2'} \in N_G(H)$ as well. Then we must have $\det(\rho_{V^H}(g_2)) = -1$, since $g_{2'}$ has odd order. Replacing $g$ by $g_2$ allows us to assume $g$ has order a power of $2$.

        Set $ K_1  := \langle H, g^2\rangle$ and $K_2 := \langle H, g\rangle$. We have $H \triangleleft K_1$, and $V^{K_1} = (V^{H})^{g^2}$. The action of $g$ has $-1$ as an eigenvalue and all other eigenvalues are either 1 or are non-real and come in conjugate pairs. The $-1$-eigenspace of $g$ acting on $V^H$ is acted trivially upon by $g^2$.
        Then the action of $g$ on $V^{K_1}$ (which is well-defined, as $g \in N_G(K_1)$) is an involution, and moreover, the eigenspaces of $-1$ for the action of $g$ on $V^H$ and $V^{K_1}$ coincide. Moreover, this eigenspace has odd dimension, since $\det(\rho_{V^H}(g)) = -1$.

        We have $g \not\in K_1$, since otherwise, $g$ would act trivially on $V^{K_1}$. Since $g^2 \in K_1$, it follows that $[K_2:K_1] = 2$. Finally, $\dim_V(K_1) - \dim_V(K_2)$ is exactly the size of the eigenspace of $-1$ associated to the action of $g$ on $V^{K_1}$, which is odd, so we are done.

        (b) $\implies$ (a): Choose $g \in K_2 \setminus K_1$. We have $g \in N_G(K_1)$, but since $[K_2:K_1] = 2$, $g^2 \in K_1$ and therefore acts trivially on $V^{K_1}$. So $g$ acts via involution on $V^{K_1}$ and fixes $(V^{K_1})^g = V^{K_2}$. Now, note the dimension of the $-1$-eigenspace associated to $g$ acting on $V^{K_1}$ is exactly $\dim_V(K_1) - \dim_V(K_2)$, which is odd, and therefore $\det(\rho_{V^{K_1}}(g)) = -1.$

    \end{proof}

    Non-orientability can be detected on the solvable subgroups of $G$.

    \begin{proposition}\label{prop:solv_cond}
        Let $V$ be a real representation of a finite group $G$. If the parity function $\pi_V$ is nonconstant on the set of solvable subgroups of $G$, then the equivalent conditions of \Cref{prop:equiv_conditions} hold for $V$. One may choose $(K_1,K_2)$ to be solvable in this case.
    \end{proposition}
    \begin{proof}
        We have that there is a solvable subgroup $K_2$ such that $\pi_V(K_2) \not\equiv \pi_V(1)$. Choose $K_2$ minimal, then $K_2$ has a normal subgroup $K_1$ of index $p$ prime. By minimality $\pi_V(K_1) = \pi_V(1)$. If $p$ is odd, then a generator $g$ of $K_2/K_1$ would act on $V^{K_1}$ with fixed point space $V^{K_2}$. The associated eigenvalues for the action are all nontrivial $p$th roots of unity, which are necessarily complex and come in conjugate pairs. Therefore $\pi_V(K_1) = \pi_V(K_2)$, a contradiction. So we are left with $p=2$, and $(K_1, K_2)$ is a pair satisfying the conditions of \Cref{prop:equiv_conditions}(b).

    \end{proof}

    \begin{conjecture}\label{q:solv_cond}
        The converse of \Cref{prop:solv_cond} holds.
    \end{conjecture}

    \begin{remark}\label{rmk:conj_hard}
        In a computational verification using GAP, we found the converse of \Cref{prop:solv_cond} holds for groups up to order 2000 and all 88 finite simple groups for which the GAP package \texttt{TomLib} \cite{TomLib1.2.11} has data. This includes 12 sporadic groups, $M_{11}, M_{12}, M_{22}, M_{23},$ $ M_{24}, J_1, J_2, J_3, HS, McL, He,$ and $Co_3$. See the appendix for all supporting code.

        The difficulty we encountered in attempts to prove \Cref{q:solv_cond} arises from non-solvable groups such as \emph{perfect groups}, those with trivial abelianization (e.g., simple groups). The question boils down to asking if, given a unit $u \in A(G)^\times$ of the form $u = \Theta_G(V)$, if $u$ has mark 1 on all solvable subgroups of $G$, then $u = [G/G]$. When $G$ is perfect, this need not hold for general units $u \in A(G)^\times$; for instance, $A(A_5)^\times$ has a unit with mark 1 at all proper subgroups of $A_5$, and mark $-1$ at $A_5$. However, this unit $u$ is not in the image of the tom Dieck homomorphism.

        Moreover, this situation is not ruled out by considering Borel--Smith conditions \cite[Chapter III, Section 5]{tD87}. The Borel--Smith conditions are conditions the dimension function $\on{dim}_V \in \on{CF}(G)$ of a real representation of a finite group $G$ must satisfy, and if $G$ is nilpotent, then they precisely characterize the image \cite{DH81}. A Borel--Smith function, a function satisfying the conditions, is determined by its value at subgroups $H$ of $G$ such that $H^{\on{ab}}$ is cyclic; however if $G$ is perfect, this includes the value at $G$ itself, so one has no control over the value at $G$. Furthermore, a Borel--Smith function in the image of $\dim\colon \on{RO}(G)\to \CF(G)$ is determined by its value at cyclic subgroups, but again, one cannot use this to rule out the possibility of a value at a non-solvable subgroup being odd when all others are even. Indeed, \Cref{prop:cyclic_detection} combined with \Cref{ex:SO_ex} demonstrates that even if $\dim_V(C)$ is the same parity for all cyclic subgroups $C$ of $G$, $\dim_V(H)$ can change parity at another subgroup $H$.
    \end{remark}

    \begin{corollary}\label{cor:odd_deg_solv}
        Let $G$ be a solvable group and $V$ a real representation of odd dimension with $V^G = 0$. Then $V$ is non-orientable.
    \end{corollary}
    \begin{proof}
        Apply \Cref{prop:solv_cond}: we have $\pi_V(1) \equiv 1 \mod 2$ but $\pi_V(G) \equiv 0$ mod 2.
    \end{proof}

    \begin{remark}
        If a group $G$ satisfies that all of its real representations of real type are odd-dimensional, then $G$ is nirrorno by \Cref{cor:odd_deg_solv}. \cite[Theorem A]{Tie15} states that the groups $G$ for which every irreducible character of real type has odd degree satisfy $O^{2'}(G)$ solvable. \cite[Theorem A]{DNT08} further states that if in addition, $G$ has that every irreducible character of quaternion type has odd degree satisfy that $P \triangleleft G$, where $P$ is a 2-Sylow. The converses do not hold however; the question of classifying such groups precisely seems difficult.
    \end{remark}

    We make a minor observation which will not be relevant in the sequel, but may be of independent interest.

    \begin{remark}
        One can also prove a proposition similar to \Cref{prop:solv_cond} but replacing `solvable' with any family of subgroups $\calF$ satisfying that the following holds: if $1\in \calF$, and if $1 \neq H \in \calF$, then there exists a $N\trianglelefteq H$ with $N \in \calF$ and $[H:N] = p$.
        An easy induction argument shows that the set of solvable subgroups of $G$, which we denote $\on{Solv}(G)$, is the largest $\calF$ for which this property holds.
    \end{remark}

    If one replaces `solvable' with `cyclic,' we find a criterion for trivial global orientation behavior, recalling \Cref{def:triv_global}.

    \begin{proposition}\label{prop:cyclic_detection}
        Let $V$ be a real representation of a finite group $G$. The parity function $\pi_V$ is constant on the set of cyclic subgroups of $G$ if and only if $V$ has trivial global orientation behavior.
    \end{proposition}
    \begin{proof}
        This follows from the formula deduced in \Cref{rmk:burnside}. Setting $H = 1$, we have $\det(\rho_{V}(g)) = |\Theta_G(V)|/|\Theta_G(V)^{\langle g\rangle }|$. Now, $\det(\rho_V(g)) = 1$ for all $g \in G$ if and only if $|\Theta_G(V)|/|\Theta_G(V)^{\langle g\rangle }|$ is constant for all $g \in G$ if and only if $\pi_V(1)/\pi_V(\langle g\rangle)$ is constant for all $g \in G$ if and only if $\pi_V(\langle g\rangle) \equiv \pi_V(1)$ for all $g \in G$.

    \end{proof}

    \subsection{Strong nirrorno-ness}\Cref{prop:solv_cond} and \Cref{q:solv_cond} motivate the following definition.

    \begin{definition}\label{def:strong_nirrorno}
        We say an even order finite group $G$ is \emph{strong nirrorno} if the following condition holds: the set \[\{\pi_\bbR = 1\}\cup \bigcup_{V \in \calG(G)}\{\pi_V\}\subseteq \CF(\on{Solv}(G), \bbF_2)\] restricted to the solvable subgroups of $G$ is linearly independent (recall $\pi_V$ is Galois invariant by \Cref{prop:galois_invariance}). By \Cref{prop:solv_cond}, $G$ strongly nirrorno implies $G$ nirrorno.
    \end{definition}

    Strong nirrorno-ness allows us to easily characterize all non-orientable representations.

    \begin{proposition}\label{prop:all_non_orientables}
        Suppose a finite group $G$ is strong nirrorno. Then every non-orientable representation is of the form \[U \oplus \bigoplus_{V \in \mathsf{nir}(G)} a_V \cdot V,\] where $U$ is a sum of real representations of complex and quaternion type and trivial representations, each $a_V \in \bbN_{\geq 0}$, and at least one sum of $a_V$'s over a Galois conjugacy class of $\mathsf{nir}(G)$ is odd.
    \end{proposition}
    \begin{proof}
        First, $G$ strong nirrorno implies that satisfies the unrestricted set of parity functions \[\{\pi_\bbR = 1\}\cup \bigcup_{V \in \calG(G)}\{\pi_V\}\subseteq \CF(G, \bbF_2)\] is linearly independent. This fact combined with \Cref{prop:galois_invariance} asserts that two $V_1, V_2 \in \mathsf{nir}(G)$ have the same parity function if and only if they are Galois-conjugate. Therefore, from the assumed linear independence, the parity function of the representation is nonzero if and only if at least the sum of coefficients for at least one Galois conjugacy class of $\mathsf{nir}(G)$ is odd. Since $G$ is also strongly nirrorno, we have that the parity function is nonzero if and only if it is nonzero restricted to $\on{Solv}(G)$. Therefore, \Cref{prop:solv_cond} implies the representation is non-orientable, as desired.

    \end{proof}

    Though the tom Dieck homomorphism (which identifies with the parity homomorphism $\pi$) has been considered in detail \cite{MM83, Tor84, Yos90, Y05}, it seems that the question of linear independence was not considered. We find by brute force that being strong nirrorno is not equivalent to being nirrorno. In fact, we determine that the parity functions need not be linearly independent even before restricting to solvable subgroups.

    \begin{theorem}\label{thm:str_nir_not_nir}
        The solvable group $G:= C_3 \rtimes \on{GL}_2(\bbF_3)$, \emph{\texttt{SmallGroup(144,125)}} is nirrorno, but not strongly nirrorno.
    \end{theorem}
    \begin{proof}
        This is obtained from explicit computation using \cite{GAP4}, see \Cref{sec:parity_fns} for all code and output. In short, $G$ is nirrorno, as it is solvable and all 12 of its nirrors $V$ satisfy $\pi_V \not\equiv 0$. However, the span of the 13 parity functions (associated to the 12 nirrors and $\bbR$) only has $\bbF_2$-dimension 12.

    \end{proof}

    \section{Non-orientability for 2-groups}

    We answer \Cref{q:the_question} in the affirmative for any $2$-group $G$. First, recall (see \cite[Chapter 5, Theorem 4.10]{Gor68} or \cite[Section 9.3]{Bou10}) that for any prime $p$, if a $p$-group $G$ does not contain a normal subgroup isomorphic of $p$-rank at least 2, i.e., $G$ has \emph{normal $p$-rank 1}, then $G$ is either a cyclic, dihedral (of order at least 16), semidihedral, or quaternion group.

    \begin{lemma}\label{lem:norm_p_rk_1}
        Let $G$ be a cyclic, dihedral, semidihedral, or quaternion $2$-group. Then every faithful nirror of $G$ is non-orientable.
    \end{lemma}
    \begin{proof}
        The following can be checked explicitly, see e.g. \cite[Section 9.3]{Bou10}.
        \begin{itemize}
            \item If $G = C_2$, then its unique faithful nirror is the sign representation, which is non-orientable. If $G = C_{2^n}$ for $n > 1$, then $G$ contains no faithful nirrors, as its faithful irreducible representation has complex type.
            \item If $G = D_{2^n}$ for $n \geq 4$, then the faithful nirrors are 2-dimensional reflection representations, which are not orientation-preserving.
            \item If $G = SD_{2^n}$ for $n \geq 4$, then $G$ has no faithful nirrors, as the only faithful irreducible real representations of $G$ are of complex type. Indeed, $G$ has irreducible faithful complex representations of degree 2 whose character values live in $\bbQ(\zeta)/\bbQ$ for $\zeta$ a primitive $2^{n-1}$th root of unity; explicitly, $\zeta - \overline{\zeta}$ is a character value.
            \item If $G = Q_{2^n}$ for $n \geq 3$, the faithful irreducible complex representations of $G$ are of quaternion type, so there are no faithful nirrors of $G$.
        \end{itemize}
    \end{proof}

    The crux of the matter is that any 2-group with normal $p$-rank at least 2 necessarily answers \Cref{q:the_question} in the alternative.

    \begin{theorem}\label{thm:2_grp_nirrorno}
        Let $G$ be a finite 2-group. Then $G$ is nirrorno.
    \end{theorem}
    \begin{proof}
        We act by induction on $|G|$; the base cases of $|G| = 1$ is vacuous and $|G|=2$ is handled in \Cref{lem:norm_p_rk_1}. Let $V$ be a nirror of $G$ and write $\rho_V$ for its associated orthogonal representation. We may assume $V$ is faithful by replacing $G$ with $G/\ker(V)$.  If $\det(\rho_)V$ is nontrivial, we are done, so assume $\rho_V\colon G \to O(V)$ has image in $SO(V)$.

        Since $V$ is of real type, $\End_{\bbR G}(V) \cong \bbR$. Since $G$ is nilpotent, it has nontrivial center $Z := Z(G)$, and therefore has normal $p$-rank at least 1. If $G$ has normal $p$-rank exactly 1, the result follows from \Cref{lem:norm_p_rk_1}, so we assume $G$ has normal $p$-rank at least 2. By centrality, any $z \in Z$ induces a $\bbR G$-module homomorphism of $V$, hence the action of $z$ is $\pm I$. However, since $V$ is faithful, we must have $|Z| = 2$ and the unique nontrivial element $z \in Z$ is $\rho_V(z) = -I$.

        Now, since $G$ has a normal subgroup $N \cong V_4$, we may write $N = \langle a, z\rangle$ with $a$ noncentral of order 2. For any $g \in G$, ${}^g a \in \{a, az\}$, and since $a$ is noncentral, there exists $g \in G$ such that ${}^ga = az$. Set $C = C_G(a)$, then $[G:C] = 2$ by the orbit-stabilizer theorem. Now $\rho_V(a)$ is an involution necessarily different than $\rho_V(z) = -I$. We obtain an eigenspace decomposition \[V = V^+\oplus V^-\] associated to $\rho_V(a)$ and its eigenvalues (which are all $\pm 1$). It follows that $C$ preserves $V^+$ and $V^-$ by centrality, but any $g \in G\setminus C$ induces isomorphisms $V^+ \to V^-$ and $V^- \to V^+$.

        First, we claim that $V^+$ is an irreducible real $C$-representation of real type. Indeed, if there exists a nonzero $\bbR C$-submodule $W \subsetneq V^+$, then for $g \in G \setminus C$, we have that $W \oplus gW$ is a proper $\bbR G$-submodule of $V$, contradicting irreducibility of $V$. Then, one may check that we obtain a bijection $\End_{\bbR C}(V^+) \to \End_{\bbR G}(V) \cong \bbR$ given by sending a homomorphism $f \in \End_{\bbR C}(V^+)$ to the homomorphism $f \oplus {}^gf \in \End_{\bbR G}(V)$, with inverse given by restriction to $C$ then projection onto $V^+$. Thus the claim is shown.

        Now $a \in N$ acts trivially on $V^+$ by construction, so $V^+$ is an irreducible real $C/\langle a\rangle$-representation of real type. By induction, $\overline{C} := C/\langle a\rangle$ is nirrorno, so there exists a $\overline{H} \leq \overline{C}$ and $\overline{g} \in N_{\overline{C}}(\overline{H})$ for which $\det(\rho_{V^+}(\overline{g})) = -1$. Let $H \leq G$ be the preimage of $\overline{H}$ and $g \in C$ be a lift of $\overline{g}$. Then $g \in N_C(H)$. As $a \in H$ acts as $-1$ on $V^-$ by construction, we have $(V^-)^H = 0$. Therefore, \[V^H = (V^+)^H \oplus (V^-)^H = (V^+)^H,\] and the deflation of $\Res^G_C (V^+)^H$ to $N_{\overline{C}}(\overline{H})$ is exactly $(V^+)^{\overline{H}}$, so \[\det(\rho_{V^H}(g)) = \det(\rho_{(V^+)^H}(g)) = \det(\rho_{(V^+)^{\overline{H}}}(\overline{g})) = -1,\] as desired.

    \end{proof}

    % In fact, there is another way to show this, but indirectly. We assume familiarity with the rational Schur index $m_\bbQ(\chi)$ of a complex representation $\chi$, see \cite[Section 74]{CR81}

    % \begin{theorem}
        %     Let $G$ be a 2-group and let $\chi$ be an irreducible complex character of real type. Then $m_\bbq(\chi) = 1$, hence \[\sum_{\sigma \in \Gal(\bbQ(\chi)/\bbQ)} {}^\sigma \chi\] is the character of a rational irreducible representation of $G$.
        % \end{theorem}
    % \begin{proof}
        %     By The Brauser--Speiser theorem \cite[74.27]{CR81}, $m_\bbQ(\chi) \in \{1,2\}$. Suppose for contradiction $m_\bbQ(\chi) = 2$. Let $A$ be the simple constituent of $\bbQ(\chi)G$ associated to $\chi$, which necessarily is a quaternion algebra over $k = \bbQ(\chi)$. Then $k$ is a subfield of $\bbQ(\zeta_{2^n})$ for some sufficiently large $n$. Since $\bbQ(\zeta_{2^n})/\bbQ$ is a cyclic extension, $k/\bbQ$ is also a cyclic extension. The class of $[A]$ in the Brauer group has order 2
        % \end{proof}

    \begin{corollary}\label{cor:basis_of_burnside}
        Let $G$ be a finite $p$-group. Then \[\{-[G/G]\}\cup \bigcup_{V \in \calG(G)}\{\pi_V\}\] is a basis of $A(G)^\times$. In particular, $G$ is strongly nirrorno.
    \end{corollary}
    \begin{proof}
        The tom Dieck homomorphism $\Theta_G$ is surjective by \cite{Tor84, Y05}. However, any complex or quaternion type irreducible real representation belongs to $\ker(\Theta_G)$, as it splits into a sum of Galois conjugates upon extension of scalars to $\bbC$, hence its dimension function is even-valued. Moreover, by \Cref{prop:galois_invariance}, the difference of two Galois conjugates belongs to $\ker(\Theta_G)$. Therefore, it suffices to show that the images of the $\pi_V$ with $V \in \calG(G)$ along with $\pi_\bbR = -[G/G]$ are linearly independent.

        This follows by dimension count and \cite[Theorem 8.5]{Bou07}, which states that the $\bbF_2$-rank of $A(G)^\times$ is the number of rational irreducible representations of $G$ whose type is trivial, cyclic of order 2, or dihedral. We claim that this is equal to $|\calG(G)| + 1$, and explain.

        Bouc's modification of the Roquette theorem \cite[Theorem 9.4.1]{Bou10} gives that every rational irreducible representation of $G$ is of the form $\Ind\Inf^G_{T/S}V$ for a subquotient $T/S$ of $G$ isomorphic to a group of normal $p$-rank 1 and $V$ the unique faithful rational representation of $T/S$. The \emph{type} of $\Ind\Inf^G_{T/S}V$ is $T/S$, and this type is unique. However, given a rational irreducible representation $\Ind\Inf^G_{T/S}V$ of $G$, it follows from \Cref{lem:norm_p_rk_1} that $V' := \bbR \otimes_\bbQ \Ind\Inf^G_{T/S}V$ is a Galois conjugacy sum (possibly with multiplicity) of real type irreducible real representations if and only if $T/S$ is trivial, cyclic of order 2, or dihedral of order at least 16. Here we use \cite[Theorem 9.4.1(b)]{Bou10}, which gives an isomorphism \[\End_{\bbQ G}(V') \cong \End_{\bbQ [T/S]}(V).\] Indeed, if $T/S$ is not one of those, then complex or quaternion type real representations appear, in which case, by \Cref{prop:galois_invariance}, all irreducible summands are complex or quaternion. On the other hand, if $T/S$ is one of these types, then $\bbR \otimes_\bbQ V$ is a Galois conjugacy sum of faithful real irreducibles, which are of real type from \Cref{lem:norm_p_rk_1} and \Cref{prop:galois_invariance}. Hence $V'$ is as well.

        Therefore, every Galois orbit of a irreducible real representation of real type corresponds to an rational irreducible representation of type 1, $C_2$, or dihedral. By dimension count, the images of the $\pi_V$ with $V \in \calG(G)$ and $\pi_\bbR$ are linearly independent, as desired.

    \end{proof}

    \begin{remark} We have a few notes about \Cref{cor:basis_of_burnside}.\label{rmk:2_grp_rmks}
        \begin{enumerate}
            \item In fact, the basis of \Cref{cor:basis_of_burnside} is exactly Bouc's basis of $A(G)^\times$ given in \cite{Bou07}. This follows from the theory of genetic bases for rational $p$-biset functors, $p$-biset functors which have an analogous `Roquette theorem.' See \cite[Part III]{Bou10} for details; the key point here is that Bouc's basis elements also correspond to subquotients of type dihedral, $C_2$, or 1.
            \item The generative AI model used while researching this paper gave an alternative proof of \Cref{cor:basis_of_burnside} relying on Schur indices, the Ritter--Segal theorem \cite{Rit72,Seg72} and Galois-invariance.
            \item For any finite $p$-group $G$, Bouc--Yal\c{c}in give a short exact sequence \cite[Corollary 1.5]{BoYa07} \[0 \to \on{Tor}_{\bbF_2} D^\Omega(G) \to \bbF_2\on{CF}_b(G) \to A(G)^\times \to 0,\] where $D^\Omega(S)$ denotes the subgroup of the Dade group of a $p$-group generated by relative syzygies (see \cite[Section 12.6]{Bou10}). One can see the real, complex, or quaternion data consequently effect the classification of 2-torsion in \cite[Corollary 12.8.11]{Bou10}: dihedral terms only correspond to non-torsion elements, while semidihedral, quaternion, and cyclic of order at least 3 all contribute to torsion.
        \end{enumerate}

    \end{remark}

    \section{A reduction for odd-index normal subgroups}

    Using Clifford theory, we can reduce the question of $G$ strong nirrorno to $N$ strong nirrorno when $[G:N]$ is odd. For most of this section, we will be considering complex characters. In this case, a nirror corresponds to a nontrivial irreducible (complex) character with Frobenius--Schur indicator 1. The first lemma is folklore and attributed to Burnside.

    \begin{lemma}\label{lem:burnside_lemma}
        (Burnside) A group $G$ of odd order has no nontrivial irreducible real-valued complex characters. In particular, every nontrivial irreducible real representation of $G$ is of complex type.
    \end{lemma}
    \begin{proof}
        Follows from \cite[Problem 3.16]{Isa76}.
    \end{proof}

    We also list a few convenient lemmas for real characters. For $\sigma\in \Gal(\bbQ^{\on{ab}}/\bbQ)$, we denote by $\Irr_\sigma(G)$ the irreducible characters of $G$ fixed by $\sigma$. For $N \trianglelefteq G$ and $\theta \in  \Irr(N)$, \[\Irr(G\mid \theta) = \{\chi \in \Irr(G)\mid \langle \Res^G_N \chi, \theta \rangle\neq 0\}, \] and say $\chi$ \emph{lies over} $\theta$. Given a subset $I \subseteq \Irr(N)$, set \[\Irr(G\mid I) = \bigcup_{\theta \in I}\Irr(G\mid \theta).\]

    \begin{lemma}\cite[Lemma 3.2]{MN23}\label{prop:malle_navarro}
        Let $N \trianglelefteq G$ be finite groups. Suppose $G/N$ has order not divisible by a prime $p$, and let $\sigma \in \Gal(\bbQ_{\zeta_{|G|}}/\bbQ)$ have order a power of $p$.
        \begin{enumerate}
            \item If $\chi \in \Irr_\sigma(G)$, then all irreducible constituents of $\Res^G_N\chi$ are $\sigma$-invariant;
            \item Suppose $p = 2$ and $\sigma$ complex-conjugates odd-order roots of unity. Suppose $\theta \in \Irr_\sigma(N)$. Then there exists a unique $\chi\in \Irr_\sigma(G)$ over $\theta$. Furthermore, if $\theta$ is $G$-invariant, then $\Res^G_N\chi = \theta$.
        \end{enumerate}
    \end{lemma}

    \begin{lemma}\cite[Lemma 4.1]{DNT08}\label{lem:more_nt_cliff}
        Let $N \trianglelefteq G$ with $[G:N]$ odd. If $\chi\in \Irr(G)$ is real-valued, then all irreducible constituents of $\Res^G_N \chi$ are real-valued. If $\theta \in \Irr(N)$ is real-valued, then there exists a unique real-valued $\psi \in \Irr(G|\theta)$.
    \end{lemma}

    \begin{proposition}\label{prop:cliff_lemmas}
        %Let $N\trianglelefteq G$ with $[G:N]$ odd and let $\calO \subseteq \Irr(N)$ be a $G$-orbit of a character which contains a real-valued character. Then every member of $\calO$ is real-valued,

        With the setup of \Cref{lem:more_nt_cliff}, if in addition, $\chi$ is nontrivial of real type, then all irreducible constituents of $\Res^G_N \chi$ are also nontrivial of real type. Moreover, \[\Res^G_N \chi = \sum_{\theta \in \calO} \theta,\] with $\calO\subseteq \Irr(N)$ a unique $G$-conjugacy class.

    \end{proposition}
    \begin{proof}
        First, let $\theta$ be an irreducible constituent of $\Res^G_N \chi$.
        \Cref{lem:more_nt_cliff} shows $\chi\in \Irr(G\mid \theta)$ is unique. Let $I$ denote the inertia group of $\theta$, i.e., $I = \{g\in G \mid {}^g\theta = \theta\}$. Then \Cref{prop:malle_navarro}(b) obtains the unique real-valued $\psi \in \Irr(I\mid \theta)$ which satisfies $\Res^I_N \psi = \theta$.
        \cite[Theorem 6.11]{Isa76} now implies $\chi = \Ind^G_I \psi$ and \[\langle \Res^G_N \chi, \theta\rangle = \langle \Res^I_N \psi, \theta\rangle = 1.\] Thus by Clifford theory, \[\Res^G_N \chi = \sum_{\theta' \in \calO}\theta', \] where $\calO$ is the $G$-orbit of $\theta$.

        We have $\nu(\theta) \in \{\pm1\}$ by the conclusion of \Cref{lem:more_nt_cliff}. If $\nu(\theta) = -1$, then $\theta$ occurs with even multiplicity in the character of every real representation of $N$, but $\langle \Res^G_N \chi, \theta\rangle = 1$ is odd, a contradiction. Hence $\nu(\theta) = 1$. It is easy to see all $G$-conjugates are also real type. Now if $\calO$ contains the trivial character $1_N$, then $\calO = \{1_N\}$. But then $\theta$ is then a nontrivial real-valued irreducible of $G/N$ which is an odd-order group, contradicting \Cref{lem:burnside_lemma}.

    \end{proof}

    \begin{theorem}\label{label:normal_2_sylow}
        Let $N \trianglelefteq G$ with $[G:N]$ odd. If $N$ is strongly nirrorno, then so is $G$. In particular, every finite group with normal Sylow 2-subgroup is strongly nirrorno.
    \end{theorem}
    \begin{proof}
        Suppose \[a_1\pi_\bbR+ \sum_{V \in \calG(G)}a_V \pi_V = 0 \in \CF(\on{Solv}(G), \bbF_2)\] with each $a_V \in \bbF_2$. It suffices to show all all $a_V$ and $a_1$ are zero. By \Cref{prop:cliff_lemmas}, we have \[\pi_V = \sum_{\theta \in \calO_V}\pi_{W_\theta},\] where $W_\theta$ denotes the real $N$-representation associated to $\theta$, and $\calO_V$ denotes the unique conjugacy class of characters under $\chi_V$. Therefore, after restricting to the solvable subgroups of $N$, we have \[a_1 \pi_\bbR + \sum_{V \in \calG(G)}\sum_{\theta \in \calO_V} a_{V}\pi_{W_\theta} = 0 \in \CF(\on{Solv}(N),\bbF_2),\] where $\calO_V$ denotes the collection of characters living under $\chi_V$ as in \Cref{prop:cliff_lemmas}.

        From \Cref{prop:cliff_lemmas}, we know that every irreducible $W_\theta$ is counted at most once in the sum, however Galois classes may appear multiple times.
        If $W_\theta$, $W_{\theta'}$ belong to the same Galois orbit, then $\theta = {}^\sigma\theta'$ for some Galois automorphism $\sigma$, and there exists a unique $\chi$ above $\theta$ and ${}^\sigma\theta'$. So $\chi{}^\sigma$ lives above $\theta'$. But since we iterate over Galois classes in $\calG(G)$, this forces $\chi^\sigma= \chi$. In particular, any two Galois-conjugate $W_\theta$, $W_{\theta'}$ in the sum expansion live under the same irreducible $\chi \in \calG(G)$.

        For $\theta\in \calO_V$, $\chi_V$ is the unique irreducible above $\theta$ by \Cref{lem:more_nt_cliff}. We write $ u(\theta) := V$, and we let $c(\theta)$ count the number of Galois conjugates of $\theta$ in $\calO_{u(\theta)}$ mod 2. Then by disjointness, we have
        \[a_1\pi_\bbR + \sum_{\theta \in \calG(N)} c(\theta)a_{u(\theta)}\pi_{W_\theta} = 0 \in \CF(\on{Solv}(N),\bbF_2).\] But we have assumed $N$ is strongly nirrorno, therefore $c(\theta)a_{u(\theta)}$ is zero for all $\theta$ appearing under some $\chi_V$ with $V \in \calG(G)$. Additionally, $a_1$ is zero.

        We claim $a_V$ must be zero for any $V \in \calG(G)$. Since $[G:N]$ is odd, $|\calO_V|$ is odd, and therefore, there exists a $\theta \in \calO_V$ for which $c(\theta) \equiv 1 \mod 2$. But we have $c(\theta)a_{u(\theta)} = c(\theta)a_V$ is 0, hence $a_V\equiv 0 \mod 2$ and we are done.

    \end{proof}

    This reduction proof is the reason the strong nirrorno criterion is useful: it is unclear to us whether an analogous reduction statement holds for nirrorno-ness.

    \section{Not all groups are nirrorno}

    In this section, we show there exists an orientable nirror for the 2-nilpotent (i.e., finite group $G$ with a normal subgroup $N$ with $[G:N]$ a power of 2) group $G := C_7 \rtimes SD_{16}$. The following lemma is motivated from \cite{KM90}.

    \begin{lemma}\label{lem:ind_lem}
        Let $d > 1$ be odd, set $C := C_d$, and let $\theta \in \Irr(C)$ be a faithful (complex) character. Let $P$ be a finite group with a homomorphism $\alpha\colon P \to \Aut(C)$, let $P_0 = \ker(\alpha)$, and let $H := C \rtimes P$ be defined by the action of $\alpha$. Let $\varphi \in \Irr(P_0)$, and define \[\chi := \Ind^H_{C\times P_0}(\theta \boxtimes \varphi).\]
        The following hold:
        \begin{enumerate}
            \item $\chi$ is irreducible of degree $|\im(\alpha)| \cdot \varphi(1)$;
            \item $\nu(\chi) = 0$ unless the involution $-1 \in \im(\alpha)$, in which case, for any $\tau \in \alpha\inv(-1)$, \[\nu(\chi) = \nu_\tau(\varphi) := \frac{1}{|P_0|}\sum_{p \in \tau P_0} \varphi(p^2).\]
        \end{enumerate}
    \end{lemma}
    \begin{proof}
        For (a), $C \times P_0$ is normal in $H$, as it is precisely the kernel of the action on $C$. It suffices to show the inertia group of $\theta \boxtimes \varphi$ by Clifford theory, and this follows since $\theta$ is faithful and by the definition of $P_0 = \ker(\alpha)$. Since $\theta$ is degree 1, as $C_d$ is abelian, and $[P:P_0] = |\im(\alpha)|$, the degree follows. (b) is a straightforward application of \cite[(4.6)]{KM90}.

    \end{proof}

    We define $G := C_7 \rtimes SD_{16}$. We write $SD_{16} := \langle a,b \mid a^8 = b^2 = 1, bab\inv = a^3\rangle$ and write $Q \leq SD_{16}$ for its unique normal subgroup isomorphic to $Q_8$. We have $Q = \langle a^2, ab\rangle$. Then $SD_{16}/Q_8$ acts on $C_7 = \langle c \mid c^7 = 1\rangle$ by inversion. Precisely, the action is given by ${}^ac = c\inv$ and ${}^bc = c\inv$. Set $I := C_7 \times Q \trianglelefteq G$ of index 2.

    Fix $\theta \in \Irr(C_7)$ faithful, and let $\varphi$ be the unique 2-dimensional irreducible character of $Q$; we have $\nu(\varphi) = -1$. Set \[\chi := \Ind^G_I(\theta \boxtimes \varphi),\] which satisfies $\chi(1) = 4$ by \Cref{lem:ind_lem}.

    \begin{theorem}\label{thm:counterex}
        The character $\chi$ is irreducible of real type. Let $V$ be the associated irreducible $\bbR G$-module of real type. Then $\dim V^K$ is even for every subgroup $K \leq G$, and consequently, $V$ is orientable.
    \end{theorem}
    \begin{proof}
        Irreducibility and real type of $\chi$ follow directly from \Cref{lem:ind_lem}: the Frobenius-Schur indicator of $\nu$ is computed as follows: \[\nu(\chi) = \nu_a(\varphi) = \frac{1}{8}\sum_{p \in aQ}\varphi(p^2) = \frac{1}{8}(4\cdot 0 + 4\cdot 2) = 1.\] Therefore $\chi$ corresponds to a real simple $\bbR G$-module of real type and dimension 4.

        We record values of $\chi$: $\chi(g) = 0$ if $g \not\in I$, and for $g_1\in C_7$, $g_2 \in P$, we have  \[\chi(g_1g_2) = (\theta(g_1)+ \theta(g_1\inv))\varphi(g_2).\]
        Since $\sum_{g\neq 1}\theta(g) = -1$, we have \[\dim V^{C_7} = \frac{1}{7}\big(\chi(1) + \sum_{g\neq 1}\chi(g)\big) = \frac{1}{7}(4 + 2(-1) + 2(-1)) = 0.\]
        Next, we consider the decomposition of $\Res^G_{SD_{16}} \chi$ into irreducibles. Recall from \Cref{lem:norm_p_rk_1} that $SD_{16}$ has no faithful irreducibles of real type, but it has a complex irreducible character of degree 2 $\psi$. Explicitly, $\psi = \Ind^{SD_{16}}_{\langle a\rangle}\lambda$, with $\lambda(a) = \zeta_8$.

        A routine Mackey formula argument using the fact that $\Res^I_Q(\theta \boxtimes \varphi) = \varphi$ gives that \[\Res^G_{SD_{16}}\chi = \Ind^{SD_{16}}_Q \varphi.\] Explicit computation shows that on $Q$, $\psi$ and $\varphi$ attain the same values. The same holds for the complex conjugate $\overline{\psi}$ and $\varphi$. Frobenius reciprocity then implies \[\langle \Ind^{SD_{16}}_Q \varphi, \psi\rangle = \langle\varphi,\varphi\rangle = 1\] and similarly for $\overline{\psi}$. Thus \[\Res^G_{SD_{16}}\chi = \psi + \overline{\psi}.\] Alternatively, one could compute this directly via character table.

        We are now ready to show that all fixed-point dimensions of $V$ are even. Let $H \leq G$. If $H$ contains $C_7$, then $V^H = V^{C_7} = 0$. Otherwise, $H$ is up to conjugacy a subgroup of $SD_{16}$. Then, \[\dim V^H = \langle \Res^{SD_{16}}_H(\psi + \overline{\psi}), 1)\rangle_\bbR = 2\langle \Res^{SD_{16}}_H\psi, 1\rangle_\bbC = 2 \dim \psi^H.\] Finally, $\dim V = 4$, and we are done. That $V$ is orientable now follows from \Cref{prop:equiv_conditions}.

    \end{proof}

    \begin{remark}
        In fact, this group is not a minimal example: a GAP computation later verified that the groups of smallest order that are not nirrorno are \texttt{SmallGroup(48,15), SmallGroup(48,17),} and \texttt{SmallGroup(48,41)}, which are all 2-nilpotent. This can be verified using the GAP code provided in the appendix; the characters in question arise in a similar manner. We also verified that of the 88 finite simple groups with data in \texttt{TomLib}, the only that are non-nirrorno are $J_2, J_3$, and $\on{PSp}_4(\bbF_5)$. A possibly interesting question would be understanding these novel faithful orientable nirrors, as their existence cannot be justified in the same manner as the example above.
    \end{remark}

    \appendix

    \section{Checking linear independence of parity functions}\label{sec:parity_fns}

    The following GAP code checks whether the parity functions of all irreducible real representations are linearly independent for all groups up to a fixed order, thus providing a proof of \Cref{thm:str_nir_not_nir}. This was written with generative AI assistance. We omit groups of odd order and 2-groups, since these cases are already known to be strong nirrorno.

    \begin{verbatim}
        ##################################################################
        ##
        ##  orientation_parity.g
        ##
        ##  DimensionFunctionsIndependent(G[, discardZero])
        ##  computes the mod 2 dimension functions of all irreducible
        ##  real representations of G, removes duplicates, and returns
        ##  true iff the remaining functions are linearly independent
        ##  in CF(G, F_2), false otherwise.
        ##

        ##################################################################
        ##
        #F  RealIrreducibleCharacters( G )
        ##
        ##  Returns the list of characters of the irreducible real
        ##  representations of G.
        ##
        RealIrreducibleCharacters := function(G)
        local tbl, irr, ind, seen, result, i, cc;

        tbl := CharacterTable(G);
        irr := Irr(tbl);
        ind := Indicator(tbl, 2);       # Frobenius-Schur indicators
        seen := [];
        result := [];

        for i in [1 .. Length(irr)] do
        if not i in seen then
        AddSet(seen, i);
        if ind[i] = 1 then                    # real
        Add(result, irr[i]);
        elif ind[i] = -1 then                 # quaternionic
        Add(result, 2 * irr[i]);
        else                                  # complex
        cc := ComplexConjugate(irr[i]);
        AddSet(seen, Position(irr, cc));  # skip
        Add(result, irr[i] + cc);
        fi;
        fi;
        od;

        return result;
        end;

        ##################################################################
        ##
        #F  DimensionOfFixedSubspace( psi, H )
        ##
        ##  The real dimension of V^H for the real
        ##  representation V with character psi
        ##  and a subgroup H of the underlying group.
        ##
        DimensionOfFixedSubspace := function(psi, H)
        return ScalarProduct(RestrictedClassFunction(psi, H),
        TrivialCharacter(H));
        end;

        ##################################################################
        ##
        #F  Mod2DimensionFunction( psi, subgroupReps )
        ##
        ##  The mod 2 dimension function of the real representation with
        ##  character psi, as a GF(2) vector indexed by the given subgroup
        ##  class representatives.
        ##
        Mod2DimensionFunction := function(psi, subgroupReps)
        return List(subgroupReps,
        H -> DimensionOfFixedSubspace(psi, H) mod 2) * Z(2)^0;
        end;

        ##################################################################
        ##
        #F  DimensionFunctionsIndependent( G[, discardZero] )
        ##
        ##  Computes the mod 2 dimension functions of all
        ##  irreducible real representations of G as vectors
        ##  in CF(G, F_2), removes duplicates (and, if
        ##  discardZero = true, also the zero function), and returns
        ##  true iff the remaining vectors are linearly independent.
        ##
        DimensionFunctionsIndependent := function(arg)
        local G, discardZero, subgroupReps, vectors;

        G := arg[1];
        discardZero := Length(arg) > 1 and arg[2] = true;

        subgroupReps := List(ConjugacyClassesSubgroups(G), Representative);

        # Set(list, func) applies func and removes duplicates
        vectors := Set(RealIrreducibleCharacters(G),
        psi -> Mod2DimensionFunction(psi, subgroupReps));

        if discardZero then
        vectors := Filtered(vectors, v -> not IsZero(v));
        fi;

        if IsEmpty(vectors) then
        return true;                 # the empty set is independent
        fi;

        return RankMat(List(vectors, ShallowCopy)) = Length(vectors);
        end;

        ###################################################################
        ##
        #F  IsSkippedOrderForIndependence( order )
        ##
        IsSkippedOrderForIndependence := function(order)
        return order mod 2 = 1
        or (order > 1 and Set(FactorsInt(order)) = [2]);
        end;

        ###################################################################
        ##
        #F  TestDimensionFunctionsIndependent( n[, discardZero] )
        ##
        ##  Runs DimensionFunctionsIndependent over every group in the
        ##  SmallGroups library of order at most n, skipping groups of odd
        ##  order and 2-groups.  Returns true if all tested groups have
        ##  linearly independent (deduplicated) mod 2 dimension functions,
        ##  false otherwise; failing groups are printed and collected.
        ##
        TestDimensionFunctionsIndependent := function(arg)
        local n, discardZero, failures, order, nr, i, G;

        n := arg[1];
        discardZero := Length(arg) > 1 and arg[2] = true;

        failures := [];
        for order in [1 .. n] do
        if IsSkippedOrderForIndependence(order) then
        Print("Skipped order ", order, " (odd order or 2-group)\n");
        else
        nr := NrSmallGroups(order);
        for i in [1 .. nr] do
        G := SmallGroup(order, i);
        if not DimensionFunctionsIndependent(G, discardZero) then
        Add(failures, [order, i]);
        Print("DEPENDENT for SmallGroup(", order, ", ", i,
        ")\n");
        fi;
        od;
        Print("Checked all ", nr, " group(s) of order ", order, "\n");
        fi;
        od;

        if IsEmpty(failures) then
        Print("\nIndependence holds for all tested groups of order <= ",
        n, "\n");
        return true;
        fi;
        Print("\nFailures at (order, id): ", failures, "\n");
        return false;
        end;
    \end{verbatim}

    The output first finds an example for groups of order 144. The following commands are used to examine \texttt{SmallGroup(144,125)}.

    \begin{verbatim}

        SetUserPreference("AtlasRep", "AtlasRepAccessRemoteFiles", false);
        Read("dimension_independence.g");
        G := SmallGroup(144, 125);;
        Print("group: ", StructureDescription(G), "\n");
        reps := List(ConjugacyClassesSubgroups(G), Representative);;
        Print("subgroup classes: ", Length(reps), "\n");
        vecs := [];;
        for psi in RealIrreducibleCharacters(G) do
        v := Mod2DimensionFunction(psi, reps);;
        Print("deg ", ValuesOfClassFunction(psi)[1], "  FS-type ",
        " vec ", List(v, x -> IntFFE(x)), "\n");
        Add(vecs, v);
        od;
        vecs := Filtered(Set(vecs), v -> not IsZero(v));;
        Print("distinct nonzero: ", Length(vecs), ", rank: ",
        RankMat(List(vecs, ShallowCopy)), "\n");
    \end{verbatim}

    The output is as follows:

    \begin{verbatim}
        group: C3 : GL(2,3)
        subgroup classes: 46
        deg 1  FS-type  vec [ 1, 1, 1, 1, 1, 1, 1, 1, 1, 1, 1, 1, 1, 1, 1, 1, 1, 1,
        1, 1, 1, 1, 1, 1, 1, 1, 1, 1, 1, 1, 1, 1, 1, 1, 1, 1, 1, 1, 1, 1, 1, 1, 1,
        1, 1, 1 ]
        deg 1  FS-type  vec [ 1, 1, 0, 1, 1, 1, 1, 1, 0, 0, 0, 1, 0, 0, 1, 0, 1, 0,
        1, 0, 1, 0, 0, 1, 0, 0, 1, 0, 0, 0, 0, 0, 1, 1, 1, 0, 0, 1, 1, 0, 0, 0, 0,
        0, 1, 0 ]
        deg 2  FS-type  vec [ 0, 0, 1, 0, 0, 0, 0, 0, 1, 0, 0, 0, 0, 1, 0, 1, 0, 0,
        0, 0, 0, 1, 1, 0, 0, 0, 0, 1, 0, 1, 0, 0, 0, 0, 0, 0, 0, 0, 0, 0, 0, 1, 0,
        0, 0, 0 ]
        deg 2  FS-type  vec [ 0, 0, 1, 0, 0, 0, 0, 0, 1, 0, 0, 0, 0, 0, 0, 0, 0, 1,
        0, 0, 0, 1, 1, 0, 0, 1, 0, 0, 0, 1, 0, 0, 0, 0, 0, 1, 1, 0, 0, 0, 0, 0, 1,
        0, 0, 0 ]
        deg 2  FS-type  vec [ 0, 0, 1, 0, 0, 0, 0, 0, 1, 1, 0, 0, 0, 0, 0, 0, 0, 0,
        0, 1, 0, 1, 1, 0, 1, 0, 0, 0, 0, 1, 0, 0, 0, 0, 0, 0, 0, 0, 0, 0, 0, 0, 0,
        1, 0, 0 ]
        deg 2  FS-type  vec [ 0, 0, 1, 0, 0, 0, 0, 0, 1, 0, 1, 0, 1, 0, 0, 0, 0, 0,
        0, 0, 0, 1, 1, 0, 0, 0, 0, 0, 1, 1, 0, 0, 0, 0, 0, 0, 0, 0, 0, 0, 1, 0, 0,
        0, 0, 0 ]
        deg 4  FS-type  vec [ 0, 0, 0, 0, 0, 0, 0, 0, 0, 0, 0, 0, 0, 0, 0, 0, 0, 0,
        0, 0, 0, 0, 0, 0, 0, 0, 0, 0, 0, 0, 0, 0, 0, 0, 0, 0, 0, 0, 0, 0, 0, 0, 0,
        0, 0, 0 ]
        deg 3  FS-type  vec [ 1, 1, 1, 1, 1, 1, 1, 1, 1, 0, 0, 1, 0, 0, 1, 0, 1, 1,
        1, 0, 0, 0, 1, 1, 0, 1, 1, 0, 0, 0, 0, 0, 1, 0, 0, 0, 1, 0, 0, 0, 0, 0, 0,
        0, 0, 0 ]
        deg 3  FS-type  vec [ 1, 1, 0, 1, 1, 1, 1, 1, 0, 1, 1, 1, 1, 1, 1, 1, 1, 0,
        1, 1, 0, 1, 0, 1, 1, 0, 1, 1, 1, 0, 1, 1, 1, 0, 0, 1, 0, 0, 0, 1, 0, 0, 0,
        0, 0, 0 ]
        deg 4  FS-type  vec [ 0, 0, 0, 0, 0, 0, 0, 0, 0, 1, 1, 0, 1, 0, 0, 0, 0, 0,
        0, 1, 0, 0, 0, 0, 0, 0, 0, 0, 0, 0, 0, 0, 0, 0, 0, 0, 0, 0, 0, 0, 0, 0, 0,
        0, 0, 0 ]
        deg 4  FS-type  vec [ 0, 0, 0, 0, 0, 0, 0, 0, 0, 1, 1, 0, 1, 1, 0, 1, 0, 0,
        0, 1, 0, 0, 0, 0, 0, 0, 0, 0, 0, 0, 1, 1, 0, 0, 0, 0, 0, 0, 0, 0, 0, 0, 0,
        0, 0, 0 ]
        deg 4  FS-type  vec [ 0, 0, 0, 0, 0, 0, 0, 0, 0, 1, 0, 0, 0, 1, 0, 1, 0, 0,
        0, 1, 0, 0, 0, 0, 0, 0, 0, 0, 0, 0, 0, 0, 0, 0, 0, 0, 0, 0, 0, 0, 0, 0, 0,
        0, 0, 0 ]
        deg 4  FS-type  vec [ 0, 0, 0, 0, 0, 0, 0, 0, 0, 0, 1, 0, 1, 1, 0, 1, 0, 0,
        0, 0, 0, 0, 0, 0, 0, 0, 0, 0, 0, 0, 0, 0, 0, 0, 0, 0, 0, 0, 0, 0, 0, 0, 0,
        0, 0, 0 ]
        deg 6  FS-type  vec [ 0, 0, 1, 0, 0, 0, 0, 0, 1, 1, 1, 0, 1, 1, 0, 1, 0, 0,
        0, 1, 0, 1, 1, 0, 1, 0, 0, 1, 1, 0, 0, 0, 0, 0, 0, 0, 0, 0, 0, 0, 0, 0, 0,
        0, 0, 0 ]
        distinct nonzero: 13, rank: 12
    \end{verbatim}

    Therefore, $G$ is not strongly nirrorno. Since $G$ is solvable, we deduce from \Cref{prop:solv_cond} that $G$ is nirrorno (noting that $G$ has 13 real irreducible representations, see e.g. \cite{GN}).

    \section{Checking Question 2}

    The following GAP code verifies that Question 2 holds for any finite simple groups which has data in the GAP library \texttt{TomLib}. There are 88 final simple groups total, including 12 sporadic groups.
    The code also checks for orientable nontrivial irreducible real representations and outputs any examples.

    \begin{verbatim}
        #############################################################################
        ##
        ##  sporadic_orientability.g
        ##
        ##  Tests the equivalence "non-orientable <=> parity function
        ##  non-constant on solvable subgroups" for groups with a library table
        ##  of marks (in particular the simple groups covered by
        ##  TomLib).
        ##
        ##  CheckTomGroup(name) returns true if the equivalence holds for the
        ##  group, false with a printed counterexample otherwise.
        ##

        Read("orientability_parity.g");

        CheckTomGroup := function(name)
        local tom, G, tbl, pc, irr, ind, n, orders, t, solv, j, o, rep,
        realpos, trivpos, constantOnes, i, psi, dims, par0, constant,
        perfdata, P, N, hom, Q, entry, Kbar, Hbar, result;

        t := Runtime();
        tom := TableOfMarks(name);
        if tom = fail then
        Print(name, ": no table of marks available\n");
        return fail;
        fi;
        G := UnderlyingGroup(tom);
        tbl := CharacterTable(name);
        if tbl = fail or FusionCharTableTom(tbl, tom) = fail then
        Print("  (no usable library table fusion; computing from group)\n");
        tbl := CharacterTable(G);
        fi;
        pc := PermCharsTom(tbl, tom);      # 1_H^G for every subgroup class
        irr := Irr(tbl);
        ind := Indicator(tbl, 2);
        n := Length(pc);
        orders := OrdersTom(tom);
        Print(name, ": |G| = ", Size(G), ", ", n, " subgroup classes, ",
        Length(irr), " irreducibles (setup ", Runtime() - t, " ms)\n");

        # solvability of the subgroup classes
        t := Runtime();
        solv := [];
        for j in [1 .. n] do
        o := orders[j];
        if o mod 2 = 1 or Length(Set(FactorsInt(o))) <= 2 then
        solv[j] := true;           # Feit-Thompson / Burnside
        else
        solv[j] := IsSolvableGroup(RepresentativeTom(tom, j));
        fi;
        od;
        Print("  solvable classes: ", Number(solv, x -> x), " of ", n,
        " (", Runtime() - t, " ms)\n");

        trivpos := Position(irr, TrivialCharacter(tbl));
        realpos := Filtered([1 .. Length(irr)],
        i -> ind[i] = 1 and i <> trivpos);
        Print("  nontrivial real-type characters: ", Length(realpos), "\n");

        # parity functions: dims via Frobenius reciprocity with 1_H^G
        constantOnes := [];
        for i in realpos do
        psi := irr[i];
        par0 := psi[1] mod 2;          # value at the trivial subgroup
        constant := true;
        for j in [1 .. n] do
        if solv[j] and ScalarProduct(psi, pc[j]) mod 2 <> par0 then
        constant := false;
        break;
        fi;
        od;
        if constant then
        Add(constantOnes, i);
        fi;
        od;
        Print("  parity-constant nontrivial real-type characters: ",
        constantOnes, "\n");

        if IsEmpty(constantOnes) then
        Print("  => equivalence holds (all nontrivial real-type parity ",
        "functions non-constant)\n");
        return true;
        fi;

        # witness scan for the parity-constant characters: pairs with
        # non-solvable H, via perfect subgroup classes from the tom.
        # this branch needs group-side class functions, so connect the
        # library table to the group now (with a fallback to computing the
        # table from the group if the class matching is ambiguous)
        if not HasUnderlyingGroup(tbl) then
        tbl := CharacterTableWithStoredGroup(G, tbl);
        fi;
        if tbl = fail then
        Print("  (class matching ambiguous; computing table from group)\n");
        tbl := CharacterTable(G);
        pc := PermCharsTom(tbl, tom);
        irr := Irr(tbl);
        ind := Indicator(tbl, 2);
        trivpos := Position(irr, TrivialCharacter(tbl));
        constantOnes := Filtered([1 .. Length(irr)], i ->
        ind[i] = 1 and i <> trivpos and
        ForAll([1 .. n], j -> not solv[j] or
        ScalarProduct(irr[i], pc[j]) mod 2 = irr[i][1] mod 2));
        Print("  parity-constant (group side): ", constantOnes, "\n");
        if IsEmpty(constantOnes) then
        return true;
        fi;
        else
        irr := Irr(tbl);
        fi;
        t := Runtime();
        perfdata := [];
        for j in [2 .. n] do
        if not solv[j] then
        P := RepresentativeTom(tom, j);
        if IsPerfectGroup(P) and Size(P) < Size(G) then
        N := Normalizer(G, P);
        hom := NaturalHomomorphismByNormalSubgroup(N, P);
        Q := Image(hom);
        entry := rec(N := N, hom := hom, Q := Q, pairs := []);
        for Kbar in List(ConjugacyClassesSubgroups(Q),
        Representative) do
        for Hbar in NormalSubgroups(Kbar) do
        if Hbar <> Kbar
        and IsCyclic(FactorGroupNC(Kbar, Hbar)) then
        Add(entry.pairs, [Kbar, Hbar]);
        fi;
        od;
        od;
        if not IsEmpty(entry.pairs) then
        Add(perfdata, entry);
        fi;
        fi;
        fi;
        od;
        Print("  perfect-class witness data: ", Length(perfdata),
        " normalizer quotients (", Runtime() - t, " ms)\n");

        result := true;
        for i in constantOnes do
        if HasNonSolvableWitness(irr[i], perfdata) then
        Print("  COUNTEREXAMPLE: Irr[", i, "] of degree ", irr[i][1],
        " is non-orientable with constant parity function!\n");
        result := false;
        fi;
        od;
        if result then
        Print("  => equivalence holds\n");
        fi;
        return result;
        end;

    \end{verbatim}

    \bibliography{bib}
    \bibliographystyle{alpha}

\end{document}